\documentclass{amsart}
\usepackage{amsfonts}
\usepackage{amsmath,amssymb}
\usepackage{amsthm}
\usepackage{amscd}
\usepackage{graphics}
\usepackage{graphicx}

\theoremstyle{remark}{

\newtheorem{Ex}{{\rm Example}}

}
\theoremstyle{plain}
{

\newtheorem{Thm}{Theorem}

}

\begin{document}
\title[Regions represented as foliated forms and smooth maps onto them]{Regions represented as foliated forms and natural smooth maps onto them}
\author{Naoki kitazawa}
\keywords{Regions in Euclidean spaces. Smooth, real analytic, or real algebraic (real polynomial) functions and maps. Derivatives. Implicit function theorem. \\
\indent {\it \textup{2020} Mathematics Subject Classification}: Primary~26B10, 57R45, 58C05, 58C25. Secondary~26D07.}

\address{Osaka Central Advanced Mathematical Institute (OCAMI) \\
3-3-138 Sugimoto, Sumiyoshi-ku Osaka 558-8585
TEL: +81-6-6605-3103
}
\email{naokikitazawa.formath@gmail.com}
\urladdr{https://naokikitazawa.github.io/NaokiKitazawa.html}
\maketitle
\begin{abstract}
The author is interested in regions surrounded by hypersurfaces and natural smooth maps onto them respecting the canonical projections of the unit spheres and so-called {\it special generic} maps and {\it moment} maps, more generally. We consider situations where these regions are foliated via 1-dimensional families of functions and their zero sets (smoothly).

We including the author are also interested in explicit and nice functions obtained by composing the canonical projections and their topological or combinatorial properties. This is of singularity theory of differentiable maps and applications to differential topology and various real geometry. Explicitly, here, as a new challenge, we discuss the 1st derivative of such a function and  its critical set. 

\end{abstract}
\section{Introduction.}
\label{sec:1}
\subsection{One of interest of us including the author in singularity theory of differentiable maps and applications to real geometry.}
\label{subsec:1.1}
Singularity theory of differentiable maps and applications to differential topology and various real geometry is regarded as a fundamental and fascinating fields in mathematics, especially geometry. Theory of Morse functions is of celebrated related theory. Higher dimensional studies are launched in \cite{thom, whitney} and developing mainly due to Saeki from pioneering studies \cite{saeki1, saeki2}. We are also interested in explicit construction of functions and maps, rather than knowing existence via methods of differential equations and so-called "homotopy", where of course the latter theory is fundamental and important and celebrated theory has been presented in \cite{eliashberg1, eliashberg2} for example.

Related to this, first, we are interested in the canonical projections of the unit spheres, more generally so-called {\it special generic} maps, and (smooth maps locally represented as so-called) {\it moment} maps.
Special generic maps are discussed in \cite{burletderham, furuyaporto, saeki2} for example. For moment maps, see \cite{buchstaberpanov} for example. We do not need (non-trivial) knowledge on such classes of smooth maps.

 The author has been interested in natural functions with given topological properties and combinatorial ones represented by {\it Reeb spaces}. {\it Reeb spaces} of continuous functions are natural quotient spaces of the spaces of the domains obtained by contracting connected components of preimages to single points. The author has been also interested in reconstructing natural maps with given images represented by regions in Euclidean spaces. Such story dates back to the former half of 20th century and \cite{reeb} for example, and the birth of theory of so-called Morse functions. Morse functions are explained in traditional textbooks such as \cite{milnor1, milnor2}. The author is interested in reconstructing explicit and nice smooth functions with given Reeb spaces. This is pioneered by Sharko in \cite{sharko} by constructing smooth functions which are locally of elementary polynomials on closed surfaces. He has constructed locally such functions first and glue them globally. This is followed by various people and one of first related studies is \cite{masumotosaeki}. Later, \cite{gelbukh1, gelbukh2, gelbukh3, gelbukh4, michalak} are presented. One of related contribution of the author is reconstruction respecting shapes of preimages (level sets) in addition (\cite{kitazawa1, kitazawa2,kitazawa4}) and real algebraic construction (\cite{kitazawa3, kitazawa5, kitazawa6}) and real analytic one (\cite{kitazawa7}: see also
 preprints \cite{kitazawa8, kitazawa9, kitazawa10} of the author). In real analytic cases, we can not use technique in differential (smooth) situations for global construction, such as bump functions and partitions of the unity.
We do not consider Reeb spaces in our paper, where we present this as a related important topic. 

\subsection{Fundamental terminologies, notions, and notation.}
\label{subsec:1.2}
We introduce or review fundamental terminologies, notions and notation of the present paper.

For a map $c:X \rightarrow Y$ and a subset $Z \subset X$, let $c {\mid}_Z$ denote the restriction of $X$ to $Z$.
For a topological space $X$ and its subspace $Y$, ${\overline{Y}}^X$ means the closure of $Y$ in $X$. For a topological space $X$ we can define the (topological) dimension as its topological invariant such as a space homeomorphic to a topological manifold and a so-called {\it CW complex}, we use $\dim X$. 

Let ${\mathbb{R}}^k$ denote the $k$-dimensional Euclidean space (real affine space), with $\mathbb{R}:={\mathbb{R}}^1$. This is a smooth manifold and also a Riemannian manifold with the standard Euclidean metric.

For a differentiable manifold $X$ and a point $p \in X$, $T_p X$ is used for the tangent space at $p$, which is a real vector space of dimension $\dim X$. The tangent bundle $TX=\bigcup T_p X$ of $X$ is naturally defined and in the case $X$ is a Riemannian manifold, we can have the subbundle $UTX \subset TX$ whose fiber is a $(\dim X-1$)-dimensional sphere and defined as the space of all tangent vectors of length $1$. 
For a differentiable map $c:X \rightarrow Y$, a {\it singular} point $p \in X$ of $c$ is a point where the rank of the differential ${dc}_p:T_pX \rightarrow T_{c(p)}Y$, which is a linear map, is smaller than both $\dim X$ and $\dim Y$. 
We use $S(c)$ for the set of all singular points of $c$ (the {\it singular set} of $c$).
In the case $Y$ is $0$- or $1$-dimensional, we also use "{\it critical}" instead of singular.

A {\it real algebraic} ({\it real analytic}) manifold $X \subset {\mathbb{R}}^{k_1+k_2}$ means a submanifold with no boundary represented as a union of some connected components of the zero set of a real polynomial (real analytic) map $e:{\mathbb{R}}^{k_1+k_2} \rightarrow {\mathbb{R}}^{k_1}$ with $0<k_1<k_1+k_2$ and with $e {\mid}_X$ having no singular point. The $k$-dimensional unit sphere $S^k:=\{(x_1,\ldots, x_{k+1}) \mid {\Sigma}_{j=1}^{k+1} {x_j}^2=1\} \subset {\mathbb{R}}^{k+1}$ is a $k$-dimensional real algebraic manifold. The $k$-dimensional unit disk $D^k:=\{(x_1,\ldots, x_{k}) \mid {\Sigma}_{j=1}^{k} {x_j}^2 \leq 1 \}$ is the uniquely defined smooth compact and connected submanifold of ${\mathbb{R}}^k$ whose boundary is $S^{k-1}$. We use ${\pi}_{k_1+k_2,k_1}: {\mathbb{R}}^{k_1+k_2} \rightarrow {\mathbb{R}}^{k_1}$ (${\pi}_{k_1+k_2,k_1}(x_1,x_2):=x_1$ for $x=(x_1,x_2) \in {\mathbb{R}}^{k_1} \times {\mathbb{R}}^{k_2}={\mathbb{R}}^{k_1+k_2}$) for the canonical projection.


\subsection{Our study in the present paper, reconstructing nice smooth maps with given images which are also regions in Euclidean spaces.}
\label{subsec:1.3}
We review a main ingredient of \cite{kitazawa3}. See also preprints by the author such as \cite{kitazawa6, kitazawa7, kitazawa8, kitazawa9, kitazawa10}. We do not assume related non-trivial arguments of the preprints.
Let $m \geq n \geq 1$ be integers.
Let $l>0$ be an integer. Let $\{S_j\}_{j=1}^l \subset {\mathbb{R}}^n$ ($n \geq 2$) be a family of mutually disjoint $l$ smooth submanifolds with no boundary each of which is a union of finitely many connected components of the zero set of a smooth function $f_j:{\mathbb{R}}^{n} \rightarrow \mathbb{R}$ with ${f_j} {\mid}_{S_j}$ having no critical point.
Let $D_{\{S_j,f_j\}_{j=1}^l}$ be an open subset of ${\mathbb{R}}^n$ satisfying $D_{\{S_j,f_j\}_{j=1}^l}=\{x \mid f_j(x)>0\} \bigcap U_{D_{\{S_j,f_j\}_{j=1}^l}}$ for some open neighborhood $U_{D_{\{S_j,f_j\}_{j=1}^l}}$ of the closure $\overline{D_{\{S_j,f_j\}_{j=1}^l}}^{{\mathbb{R}}^n}$ in ${\mathbb{R}}^n$ and $S_j=\{x \mid f_j(x)=0\} \bigcap U_{D_{\{S_j,f_j\}_{j=1}^l}}$.

We can define a union $X_{D_{\{S_j,f_j\}_{j=1}^l},m}:=\{(x,{(y_j)}_{j=1}^{m-n+1}) \mid x \in \overline{D_{\{S_j,f_j\}_{j=1}^l}}^{{\mathbb{R}}^n}, {\prod}_{j=1}^l (f_j(x))-{\Sigma}_{j=1}^{m-n+1} {y_j}^2=0\}$ of connected components of the zero set of ${\prod}_{j=1}^l (f_j(x))-{\Sigma}_{j=1}^{m-n+1} {y_j}^2$. 
\begin{Thm}
\label{thm:1}
$X_{D_{\{S_j,f_j\}_{j=1}^l},m}$ is an $m$-dimensional smooth submanifold of ${\mathbb{R}}^{m+1}$ with no boundary.
Especially, if $f_j$ are real analytic {\rm (}real polynomial{\rm )} functions, then the manifold is also real analytic {\rm (}resp. real algebraic{\rm )}{\rm :} hereafter, in similar situations, we can say that a fact of this type holds, unless otherwise stated. 
\end{Thm}
\begin{proof}
The proof is essentially presented first in \cite{kitazawa3}.
\ \\
Case 1-A Let $(x_0,(y_{0,j})_{j=1}^{m-n+1}) \in \{x,{(y_j)}_{j=1}^{m-n+1} \mid x \in \overline{D_{\{S_j,f_j\}_{j=1}^l}}^{{\mathbb{R}}^n}, {\prod}_{j=1}^l (f_j(x))-{\Sigma}_{j=1}^{m-n+1} {y_j}^2=0\}$ such that $x_0 \in D_{\{S_j,f_j\}_{j=1}^l}$. \\ In this case, the value of the partial derivative $\frac{\partial ({\prod}_{j=1}^l (f_j(x))-{\Sigma}_{j=1}^{m-n+1} {y_j}^2)}{\partial y_i}$ of the function ${\prod}_{j=1}^l (f_j(x))-{\Sigma}_{j=1}^{m-n+1} {y_j}^2$ is not $0$, for some $i$.\\
\ \\
Case 1-B Let $(x_0,(y_{0,j})_{j=1}^{m-n+1}) \in \{x,{(y_j)}_{j=1}^{m-n+1} \mid x \in \overline{D_{\{S_j,f_j\}_{j=1}^l}}^{{\mathbb{R}}^n}, {\prod}_{j=1}^l (f_j(x))-{\Sigma}_{j=1}^{m-n+1} {y_j}^2=0\}$ such that $x_0 \notin D_{\{S_j,f_j\}_{j=1}^l}$. \\ In this case, the value of the partial derivative $\frac{\partial ({\prod}_{j=1}^l (f_j(x))-{\Sigma}_{j=1}^{m-n+1} {y_j}^2)}{\partial x_i}$ of the function ${\prod}_{j=1}^l (f_j(x))-{\Sigma}_{j=1}^{m-n+1} {y_j}^2$ is not $0$, for some $i$. \\
\ \\
We use implicit function theorem to complete the proof.
\end{proof}
\begin{Ex}
\label{ex:1}
If we consider the case $D_{\{S_j,f_j\}_{j=1}^l}$ with $l=1$, $f_1(x_1,\ldots x_{n+1})=:1-{\Sigma}_{j=1}^{n+1} ({x_j}^2)$ and $D_{\{S_j,f_j\}_{j=1}^l}:=D^n-S^{n-1}$, then $X_{D_{\{S_j,f_j\}_{j=1}^l},m}=S^m$.
\end{Ex}
By the construction, the map ${\pi}_{m+1,n} {\mid}_{X_{D_{\{S_j,f_j\}_{j=1}^l},m}}$  is also a special generic map. Rigorously, a {\it special generic} map $c:X \rightarrow Y$ is a smooth map between manifolds with no boundary whose singular points have the form $c(x_1,\ldots x_{\dim X})=(x_1,\ldots x_{\dim Y-1},{\Sigma}_{j=1}^{\dim X-\dim Y+1} {x_{\dim Y-1+j}}^2)$ with suitable local coordinates.

\subsection{The content of the present paper.}
\label{subsec:1.4}
In the next section, we prove one of our main result (Theorem \ref{thm:2}). 
In short, we consider a foliated region $D_{\{S_j\}_{j \in J}}:=D_{\{S_j,f_j\}_{j=1}^l}$ and consider a slightly revised reconstruction in the Subsection \ref{subsec:1.3}.
We have an $m$-dimensional smooth submanifold of ${\mathbb{R}}^{m+2}$ with no boundary, instead. 
This is also regarded as a kind of our remark to \cite{kitazawa6} (\cite[Main Theorems]{kitazawa6}), and we also present a kind of related new theorems and exposition as Example \ref{ex:2} and Theorem \ref{thm:3}. Remember that we do not assume non-trival knowledge on the preprint \cite{kitazawa6}.
We also introduce related important examples as another main result in the third section. Here, as a new challenge, we discuss the {\it 1st derivative} of the function obtained by the restriction of ${\pi}_{m+2,1}$ to the obtained $m$-dimensional smooth submanifold of ${\mathbb{R}}^{m+2}$ of Theorem \ref{thm:2} and their critical points (Theorems \ref{thm:4} and \ref{thm:5}). 

\section{Our main result.}
We abuse the notation from the previous subsection. We consider a "foliated" region $D_{\{S_j,f_j\}_{j=1}^l}$. Hereafter, we use $D_{\{S_j\}_{j \in J}}$ with $J$ being a finite set, where the positive integer $l$ and the functions $f_j$ are irrelevant.

Again, $D_{\{S_j\}_{j \in J}} \subset {\mathbb{R}}^n$ is an open subset of ${\mathbb{R}}^n$, ${\overline{D_{\{S_j\}_{j \in J}}}}^{{\mathbb{R}}^n}-D_{\{S_j\}_{j \in J}}={\sqcup}_{j \in J} S_j$, and each $S_j$ is a smooth submanifold of dimension $n-1$ of ${\mathbb{R}}^n$ with no boundary and represented as a union of finitely many connected components of the zero set of some smooth function $g_j:{\mathbb{R}}^n \rightarrow \mathbb{R
}$ with $g_j {\mid}_{S_j}$ having no critical point.
\begin{Thm}
\label{thm:2}
Suppose that the region $D_{\{S_j\}_{j \in J}} \subset {\mathbb{R}}^n$ is represented in the following way.
\begin{enumerate}
\item \label{thm:2.1}
 For a smooth function $F_{D_{\{S_j\}_{j \in J}}}:{\mathbb{R}}^{n} \times I \rightarrow \mathbb{R}$, $D_{\{S_j\}_{j \in J}}=\{(x,t) \mid F_{D_{\{S_j\}_{j \in J}}}(x,t)=0\} \subset {\mathbb{R}}^{n} \times I$, where $I$ is an interval of either the form of the following with $a_1$, $a_2$ and $a$ being real numbers and called {\rm boundary points} of $I$.
\begin{itemize}
\item $I=\{a_1 \leq t \leq a_2\}$. 
\item $I=\{t \geq a\}$.
\end{itemize} 
\item \label{thm:2.2} 
For each point $t \in I$, $F_{D_{\{S_j\}_{j \in J}}}(x,t)$ canonically defines the function $F_{D_{\{S_j\}_{j \in J},t}}:{\mathbb{R}}^n \rightarrow \mathbb{R}$ {\rm (}$F_{D_{\{S_j\}_{j \in J},t}}(x):=F_{D_{\{S_j\}_{j \in J}}}(x,t)${\rm )}. Furthermore, the following are satisfied.
\begin{itemize}
\item For distinct $t_1,t_2 \in I$, ${F_{D_{\{S_j\}_{j \in J},t_1}}}^{-1}(0)$ and ${F_{D_{\{S_j\}_{j \in J},t_2}}}^{-1}(0)$ are mutually disjoint.
\item In the case $I=\{a_1 \leq t \leq a_2\}$, ${\sqcup}_{i=1}^2 {F_{D_{\{S_j\}_{j \in J},a_i}}}^{-1}(0)={\sqcup}_{j \in J} S_j$ and ${F_{D_{\{S_j\}_{j \in J},a_i}}}^{-1}(0)={\sqcup}_{j \in J_i} S_j$ for some $J_i \subset J$ with $J=J_1 \sqcup J_2$. In the case $I=\{t \geq a\}$, ${F_{D_{\{S_j\}_{j \in J},a}}}^{-1}(0)={\sqcup}_{j \in J} S_j$.
Furthermore, in the case $I=\{a_1 \leq t \leq a_2\}$, $F_{D_{\{S_j\}_{j \in J},a_i}} {\mid}_{{F_{D_{\{S_j\}_{j \in J},a_i}}}^{-1}(0)}$ has no critical point, and in the case $I=\{t \geq a\}$, $F_{D_{\{S_j\}_{j \in J},a}} {\mid}_{{F_{D_{\{S_j\}_{j \in J},a}}}^{-1}(0)}$ has no critical point.
\end{itemize}
\item \label{thm:2.3} The value of the partial derivative $\frac{\partial F_{D_{\{S_j\}_{j \in J}}}}{\partial t}$ at each $(x,t)=(x_0,t_0) \in {\mathbb{R}}^n \times I$ is always non-zero.
\end{enumerate}
Let $m \geq n \geq 1$ be an integer.
In this situation, we have an $m$-dimensional smooth submanifold $X_{F_{D_{\{S_j\}_{j \in J}}},m}$ of ${\mathbb{R}}^{m+2}$ with no boundary.
\end{Thm}
\begin{proof}
We define $X_{F_{D_{\{S_j\}_{j \in J}}},m}:=\{(x,t,{(y_j)}_{j=1}^{m-n+1}) \in {\mathbb{R}}^n \times I \times {\mathbb{R}}^{m-n+1} \mid F_{D_{\{S_j\}_{j \in J}}}(x,t)=0, G_{D_{\{S_j\}_{j \in J}}}(t,{(y_j)}_{j=1}^{m-n+1}):=(t-a_1)(a_2-t)-{\Sigma}_{j=1}^{m-n+1} {y_j}^2=0\}=\{(x,t,{(y_j)}_{j=1}^{m-n+1}) \in {\mathbb{R}}^n \times \mathbb{R} \times {\mathbb{R}}^{m-n+1} \mid F_{D_{\{S_j\}_{j \in J}}}(x,t)=0, G_{D_{\{S_j\}_{j \in J}}}(t,{(y_j)}_{j=1}^{m-n+1}):=(t-a_1)(a_2-t)-{\Sigma}_{j=1}^{m-n+1} {y_j}^2=0\}$ in the case
 $I=\{a_1 \leq t \leq a_2\}$ and $X_{F_{D_{\{S_j\}_{j \in J}}},m}:=\{(x,t,{(y_j)}_{j=1}^{m-n+1}) \in {\mathbb{R}}^n \times I \times {\mathbb{R}}^{m-n+1} \mid F_{D_{\{S_j\}_{j \in J}}}(x,t)=0, G_{D_{\{S_j\}_{j \in J}}}(t,{(y_j)}_{j=1}^{m-n+1}):=(t-a)-{\Sigma}_{j=1}^{m-n+1} {y_j}^2=0\}=\{(x,t,{(y_j)}_{j=1}^{m-n+1}) \in {\mathbb{R}}^n \times \mathbb{R} \times {\mathbb{R}}^{m-n+1} \mid F_{D_{\{S_j\}_{j \in J}}}(x,t)=0, G_{D_{\{S_j\}_{j \in J}}}(t,{(y_j)}_{j=1}^{m-n+1}):=(t-a)-{\Sigma}_{j=1}^{m-n+1} {y_j}^2=0\}$ in the case
 $I=\{t \geq a\}$. We prove that this is our desired case. \\
\ \\
Case 2-A $(x,t,{(y_j)}_{j=1}^{m-n+1}) \in X_{F_{D_{\{S_j\}_{j \in J}}},m}$, where  $t$ is not a boundary point of $I$. \\
The value of the derivative $\frac{\partial {F_{D_{\{S_j\}_{j \in J}}}}}{\partial t}$ at the point is not $0$, by the assumption. That of the derivative $\frac{\partial  G_{D_{\{S_j\}_{j \in J}}}}{\partial t}$ at the point is not $0$ except at $t=\frac{a_1+a_2}{2}$ in the case $I=\{a_1 \leq t \leq a_2\}$.
The value of the derivative $\frac{\partial {F_{D_{\{S_j\}_{j \in J}}}}}{\partial y_j}$ is naturally defined and always $0$. That of the derivative $\frac{\partial G_{D_{\{S_j\}_{j \in J}}}}{\partial y_{j_0}}$ for some $y_{j_0}$ at the point is not $0$. \\
\ \\
Case 2-B  $(x,t,{(y_j)}_{j=1}^{m-n+1}) \in X_{F_{D_{\{S_j\}_{j \in J}}},m}$, where  $t$ is a boundary point of $I$. \\
The value of the derivative $\frac{\partial {F_{D_{\{S_j\}_{j \in J}}}}}{\partial t}$ at the point is not $0$, by the assumption. The value of the derivative $\frac{\partial {F_{D_{\{S_j\}_{j \in J}}}}}{\partial x_{j_0}}$ at the point is not $0$ for some $x_{j_0}$ by the assumption that $S_j$ is a smooth submanifold of dimension $n-1$ with no boundary and a union of finitely many connected components of the zero set of the smooth function $F_{D_{\{S_j\}_{j \in J},a_i}}$ ($F_{D_{\{S_j\}_{j \in J},a}}$) with the restriction to ${F_{D_{\{S_j\}_{j \in J},a_i}}}^{-1}(0)$ (resp. ${F_{D_{\{S_j\}_{j \in J},a}}}^{-1}(0)$) having no critical point. The value of the derivative $\frac{\partial  G_{D_{\{S_j\}_{j \in J}}}}{\partial t}$ at the point is not $0$. The value of the derivative $\frac{\partial G_{D_{\{S_j\}_{j \in J}}}}{\partial x_j}$ at the point is naturally defined and always $0$.\\
\ \\
We remember Case 2-A and Case 2-B. By virtue of implicit function theorem, this completes the proof.
\end{proof}
Note that the case $I=\{t \geq a\}$ of Theorem \ref{thm:2} is a kind of our short remark to \cite{kitazawa6} (\cite[Main Theorems]{kitazawa6}). 
The case $I=\{a_1 \leq t \leq a_2\}$ is a statement of new type and essential in our new result Theorems \ref{thm:4} and \ref{thm:5}, presented in the next section.
In such a situation, Theorem \ref{thm:3} is also a theorem, remarking on \cite{kitazawa6} (\cite[Main Theorem 3]{kitazawa6}). 
\begin{Thm}
\label{thm:3}
\begin{enumerate}
\item In Theorem \ref{thm:2}, in the case $I=\{t \geq a\}$, $G_{D_{\{S_j\}_{j \in J}}}(t,{(y_j)}_{j=1}^{m-n+1}):=t-a-{\Sigma}_{j=1}^{m-n+1} {y_j}^2=0$ can be replaced by $G_{D_{\{S_j\}_{j \in J}}}(t,{(y_j)}_{j=1}^{m-n+1}):=G_{+,a}(t)-{\Sigma}_{j=1}^{m-n+1} {y_j}^2$ with a smooth function $G_{+,a}:\mathbb{R} \rightarrow \mathbb{R}$ such that $\{t \mid G_{+,a}(t)>0\}=I-\{a\}$, that $G_{+,a}(a)=0$, and that at $a$, the value of th 1st derivative is not $0$. For this, refer to and compare with \cite[Main Theorem 3]{kitazawa6}, where we do not assume related non-trivial knowledge.
\item In Theorem \ref{thm:2}, in the case $I=\{a_1 \leq t \leq a_2\}$, $G_{D_{\{S_j\}_{j \in J}}}(t,{(y_j)}_{j=1}^{m-n+1}):=(t-a_1)(a_2-t)-{\Sigma}_{j=1}^{m-n+1} {y_j}^2$ can be replaced by $G_{D_{\{S_j\}_{j \in J}}}(t,{(y_j)}_{j=1}^{m-n+1}):=G_{+,a_1,a_2}(t)-{\Sigma}_{j=1}^{m-n+1} {y_j}^2$ with a smooth function $G_{+,a_1,a_2}:\mathbb{R} \rightarrow \mathbb{R}$ such that $\{t \mid G_{+,a_1,a_2}(t)>0\}=I-\{a_1,a_2\}$, that $G_{+,a}(a_i)=0$ for $i=1,2$, and that at each $a_1$ and $a_2$, the value of the 1st derivative is not $0$. For example, consider the case $G_{+,a_1,a_2}(t):={(\frac{a_2-a_1}{2})}^2-{(t-\frac{a_1+a_2}{2})}^2$.
\end{enumerate}
\end{Thm}
\begin{proof}
We can understand this immediately from the story of our proof.  More precisely, the value of the derivative $\frac{\partial  G_{D_{\{S_j\}_{j \in J}}}}{\partial t}$ at the point $t$ which is not a boundary point of $I$ is not important.
This completes the proof.
\end{proof}
 Examples \ref{ex:2}  and \ref{ex:3} are also exposition of a new type.
\begin{Ex}
 \label{ex:2}

 Let $J$ be a one-element set $J_0=\{j_0\}$ and $S_{j_0}:=\{(x_1,x_2) \in {\mathbb{R}}^2 \mid {x_1}^2+{x_2}^2-1=0\}$. We put $I=\{t \geq a=0\}$ and $F_{D_{\{S_{j}\}_{j \in J_0}}}(x,t)={x_1}^2+{x_2}^2-(1-t)$ to have a desired case of Theorem \ref{thm:2}.

We consider the case of Example \ref{ex:1}: by sending $(x_1,x_2,{(y_j)}_{j=1}^{m-1}) \in S^m$ satisfying ${x_1}^2+{x_2}^2-(1-t)=0$ with $t \geq 0$ to $(x_1,x_2,t,{(y_j)}_{j=1}^{m-1})$, we have the present case.
\end{Ex}
\begin{Ex}
\label{ex:3}
 Let $J$ be a two-element set $J_1=\{1,2\}$ and $S_{j}:=\{(x_1,x_2) \in {\mathbb{R}}^2 \mid {x_1}^2+{x_2}^2-j^2=0\}$. We put $I=\{a_1=1 \leq t \leq a_2=2\}$ and $F_{D_{\{S_{j}\}_{j \in J_1}}}(x,t)={x_1}^2+{x_2}^2-t^2$ to have a desired case of Theorem \ref{thm:2}.

We consider the case of Theorem \ref{thm:1} with $l=2$ and $(f_1(x_1,x_2):={x_1}^2+{x_2}^2-1,f_2(x_1,x_2):=4-{x_1}^2-{x_2}^2,S_j=\{(x_1,x_2) \mid f_j(x_1,x_2)=0\})$. By sending $(x_1,x_2,{(y_j)}_{j=1}^{m-1}) \in X_{D_{\{S_j,f_j\}_{j=1}^l},m}$ satisfying ${x_1}^2+{x_2}^2-t^2=0$ with $1 \leq t \leq 2$ to $(x_1,x_2,t,{(\frac{\sqrt{(t-1)(2-t)}}{\sqrt{t^2-1}\sqrt{4-t^2}}y_j)}_{j=1}^{m-1})=((x_1,x_2,t,{(\frac{1}{\sqrt{t+1}\sqrt{t+2}}y_j)}_{j=1}^{m-1}))$, we have the present case: rigorously, note that  $(x_1,x_2,t,{(\frac{\sqrt{(t-1)(2-t)}}{\sqrt{t^2-1}\sqrt{4-t^2}}y_j)}_{j=1}^{m-1})$ is for $1<t<2$ and that $((x_1,x_2,t,{(\frac{1}{\sqrt{t+1}\sqrt{t+2}}y_j)}_{j=1}^{m-1}))$ is for $1 \leq t \leq 2$.

\end{Ex}
\section{Our additional result and examples related to Theorems \ref{thm:2} and \ref{thm:3}, (canonically defined) derivatives of the functions defined by the canonical projections.}
\begin{Thm}
\label{thm:4}
We have an example of Theorem \ref{thm:2} {\rm (}Theorem \ref{thm:3}{\rm )} as follows.
\begin{enumerate}
\item Let $n=2$.
\item Choose an arbitrary smooth function $c_{+}:\mathbb{R} \rightarrow \mathbb{R}$ whose value at each point is positive. 
\item Let $a_1<0$ be an arbitrary negative number and $a_2=1$. Let $I=\{a_1 \leq t \leq a_2\}$.
\item We define $J:=\{1,2\}$ and ${F_{D_{\{S_j\}_{j \in J}}}}(x_1,x_2,t)=x_2-tc_{+}(x_1)$. 
\end{enumerate}

We have our desired $m$-dimensional manifold $X_{F_{D_{\{S_j\}_{j \in J}}},m} \subset {\mathbb{R}}^{m+2}$ with ${\pi}_{m+2,1} {\mid}_{X_{F_{D_{\{S_j\}_{j \in J}}},m}}$ having no critical point.
\end{Thm}
\begin{proof}
By the definition, we can prove all statements here immediately.
\end{proof}

We discuss Theorem \ref{thm:4} further.

Related to this, we discuss a canonically defined derivative of ${\pi}_{m+k,1} {\mid}_{X^m}$ the manifold $X^m$ of whose domain is an $m$-dimensional smooth submanifold of ${\mathbb{R}}^{m+k}$ ($k>0$) with no boundary and which has no critical point. We can also consider a so-called ({\it positive}) {\it gradient flow}, which is a section of the tangent bundle $TX^m$ over $X^m$, which is associated to the function ${\pi}_{m+k,1} {\mid}_{X^m}$, and the value of which is a  non-zero vector, at each point of $X^m$.
For each point $p \in X^m$, we can have the uniquely defined element $v_{{\rm u},p,+} \in T_p X^m \bigcap UTX^m$ which is orthogonal to every element of $T_p X^m$ mapped to the zero vector of $T_{{\pi}_{m+k,1}(p)} \mathbb{R}$ and which is mapped to a vector of $T_{{\pi}_{m+k,1}(p)} \mathbb{R}$ parallel to a vector $a_{v_{{\rm u},p,+}} \in \mathbb{R}$ identified canonically with a positive number $a_{v_{{\rm u},p,+}}$. 
In other words, we consider the tangent vector of length one along the positive gradient flow, uniquely, and we can project it canonically. This defines the {\it 1st derivative} ${{\pi}_{m+k,1} {\mid}_{X^m}}^{\prime}:X^m \rightarrow \mathbb{R}$ as the real-valued smooth function whose value at each point $p \in X^m$ is $a_{v_{{\rm u},p,+}}>0$ and $a_{v_{{\rm u},p,+}} \leq 1$. Although we do not explain related fundamental theory or arguments rigorously, this definition is compatible with the 1st derivative of a differentiable real-valued function $c:\mathbb{R} \rightarrow \mathbb{R}$ in the classical calculus and various similar definitions.
  
\begin{Thm}
\label{thm:5}
In Theorem \ref{thm:4}, the critical set $S({{\pi}_{m+2,1} {\mid}_{X_{F_{D_{\{S_j\}_{j \in J}}},m}}}^{\prime})$ of the 1st derivative ${{\pi}_{m+2,1} {\mid}_{X_{F_{D_{\{S_j\}_{j \in J}}},m}}}^{\prime}$ is the set of all points satisfying at least one of the following four.
\begin{enumerate}
\item A point of the form $(x_1,c_{+}(x_1),1,{(0)}_{j=1}^{m-1})$ with the value of the 2nd derivative of $c_{+}$ at $x_1$ being $0$. 
\item A point of the form $(x_1,a_1c_{+}(x_1),a_1,{(0)}_{j=1}^{m-1})$ with the value of the 2nd derivative of $c_{+}$ at $x_1$ being $0$. 
\item A point of the form $(x_1,0,0,{(y_j)}_{j=1}^{m-1}) \in X_{F_{D_{\{S_j\}_{j \in J}}},m}$.
\item A point of the form $(x_1,tc_{+}(x_1),t,{(y_j)}_{j=1}^{m-1}) \in X_{F_{D_{\{S_j\}_{j \in J}}},m}$ with $x_1$ being a critical point of the function $c_{+}$.
\end{enumerate} 
In addition, a point is of the third or the fourth case above, if and only if the value of the 1st derivative ${{\pi}_{m+2,1} {\mid}_{X_{F_{D_{\{S_j\}_{j \in J}}},m}}}^{\prime}$ there is $1$.
\end{Thm}
\begin{proof}
By the coordinate of the form $(x_1,x_2,t,{(y_j)}_{j=1}^{m-1})$, the functions $F_{D_{\{S_j\}_{j \in J}}}$ and $G_{D_{\{S_j\}_{j \in J}}}$, the construction, and the definition of the 1st derivatives of our real-valued functions, 
each vector $v_{{\rm u},p,+} \in T_p X^m \bigcap UTX^m$ with $X^m:=X_{F_{D_{\{S_j\}_{j \in J}}},m}$ in the definition above is represented as a vector of length $1$ tangent to a point of some graph $\{(x_1,tc_{+}(x_1) \mid x_1 \in \mathbb{R}\}$ after it is mapped to ${\mathbb{R}}^2$, (rigorously, $T{\mathbb{R}}^2$,) by (the differential of) ${\pi}_{m+2,2} {\mid}_{X_{F_{D_{\{S_j\}_{j \in J}}},m}}$.

We can immediately see that the third and fourth cases yield critical points of ${{\pi}_{m+2,1} {\mid}_{X_{F_{D_{\{S_j\}_{j \in J}}},m}}}^{\prime}$  and that the values of ${{\pi}_{m+2,1} {\mid}_{X_{F_{D_{\{S_j\}_{j \in J}}},m}}}^{\prime}$ are $1$ there.

For the first and the second case, consider $x_1$ and $t$ and investigate behavior according to changes of each of these values, or partial derivatives. For the change of values of $x_1$, a point in the first or the second case is regarded as a critical point of ${{\pi}_{m+2,1} {\mid}_{X_{F_{D_{\{S_j\}_{j \in J}}},m}}}^{\prime}$ , by fundamental properties of the 2nd derivative of $c_{+}$ and our manifolds and maps. We can also see that in the set of all points of the form $(x_{1,0},tc_{+}(x_{1,0}),t,{(0)}_{j=1}^{m-1}) \in X_{F_{D_{\{S_j\}_{j \in J}}},m}$ with $x_{1,0}$ being fixed, which is a smooth submanifold diffeomorphic to $S^{m-1}$ in $X_{F_{D_{\{S_j\}_{j \in J}}},m}$, ${{\pi}_{m+2,1} {\mid}_{X_{F_{D_{\{S_j\}_{j \in J}}},m}}}^{\prime}$ has a local extremum at $t=a_1,a_2$ ($t=a_1,1$), by our manifolds and maps. 
We can understand that the first and second cases yield critical points of ${{\pi}_{m+2,1} {\mid}_{X_{F_{D_{\{S_j\}_{j \in J}}},m}}}^{\prime}$.  

It is sufficient to investigate a point of the form $(x_1,tc_{+}(x_1),t,{(y_j)}_{j=1}^{m-1}) \in X_{F_{D_{\{S_j\}_{j \in J}}},m}$ satisfying the following.
\begin{itemize}
\item $t \neq 0$ and $x_1$ is not a critical point of the function $c_{+}$.
\item In the case $t=a_1,a_2$ ($t=a_1,1$), the 2nd derivative of $c_{+}$ at $x_1$ is not $0$.
\end{itemize}
More precisely, we see that such a point is not a critical point of ${{\pi}_{m+2,1} {\mid}_{X_{F_{D_{\{S_j\}_{j \in J}}},m}}}^{\prime}$. 
We consider the case $t:=t_0 \neq a_1$ and $t=t_0 \neq a_2$: we can see that in the set of all points of the form $(x_{1,0},t_0c_{+}(x_{1,0}),t_0,{(0)}_{j=1}^{m-1}) \in X_{F_{D_{\{S_j\}_{j \in J}}},m}$ with $x_{1,0}$ being fixed, which is a smooth submanifold diffeomorphic to $S^{m-1}$ in $X_{F_{D_{\{S_j\}_{j \in J}}},m}$, ${{\pi}_{m+2,1} {\mid}_{X_{F_{D_{\{S_j\}_{j \in J}}},m}}}^{\prime}$ does not have a local extremum or regarded as a critical point of ${{\pi}_{m+2,1} {\mid}_{X_{F_{D_{\{S_j\}_{j \in J}}},m}}}^{\prime}$, for some local change of $t$, at the point $(x_{1,0},t_0c_{+}(x_{1,0}),t_0,{(y_j)}_{j=1}^{m-1}) \in X_{F_{D_{\{S_j\}_{j \in J}}},m}$. In the case $t:=t_0=a_1,a_2$ with $x_1:=x_{1,0}$, under the conditions, for some local change of values of $x_1$ (at $x_{0,1}$), the point $(x_{1,0},t_0c_{+}(x_{1,0}),t_0,{(y_j)}_{j=1}^{m-1}) \in X_{F_{D_{\{S_j\}_{j \in J}}},m}$ can not be regarded as a critical point of ${{\pi}_{m+2,1} {\mid}_{X_{F_{D_{\{S_j\}_{j \in J}}},m}}}^{\prime}$.
 
We can also understand the following immediately by the arguments: a point is of the third or the fourth case in the statement, if and only if the value of ${{\pi}_{m+2,1} {\mid}_{X_{F_{D_{\{S_j\}_{j \in J}}},m}}}^{\prime}$ there is $1$.

See also Figure \ref{fig:1}.

This completes the proof.
\end{proof}

\begin{figure}
\includegraphics[width=40mm, height=40mm]{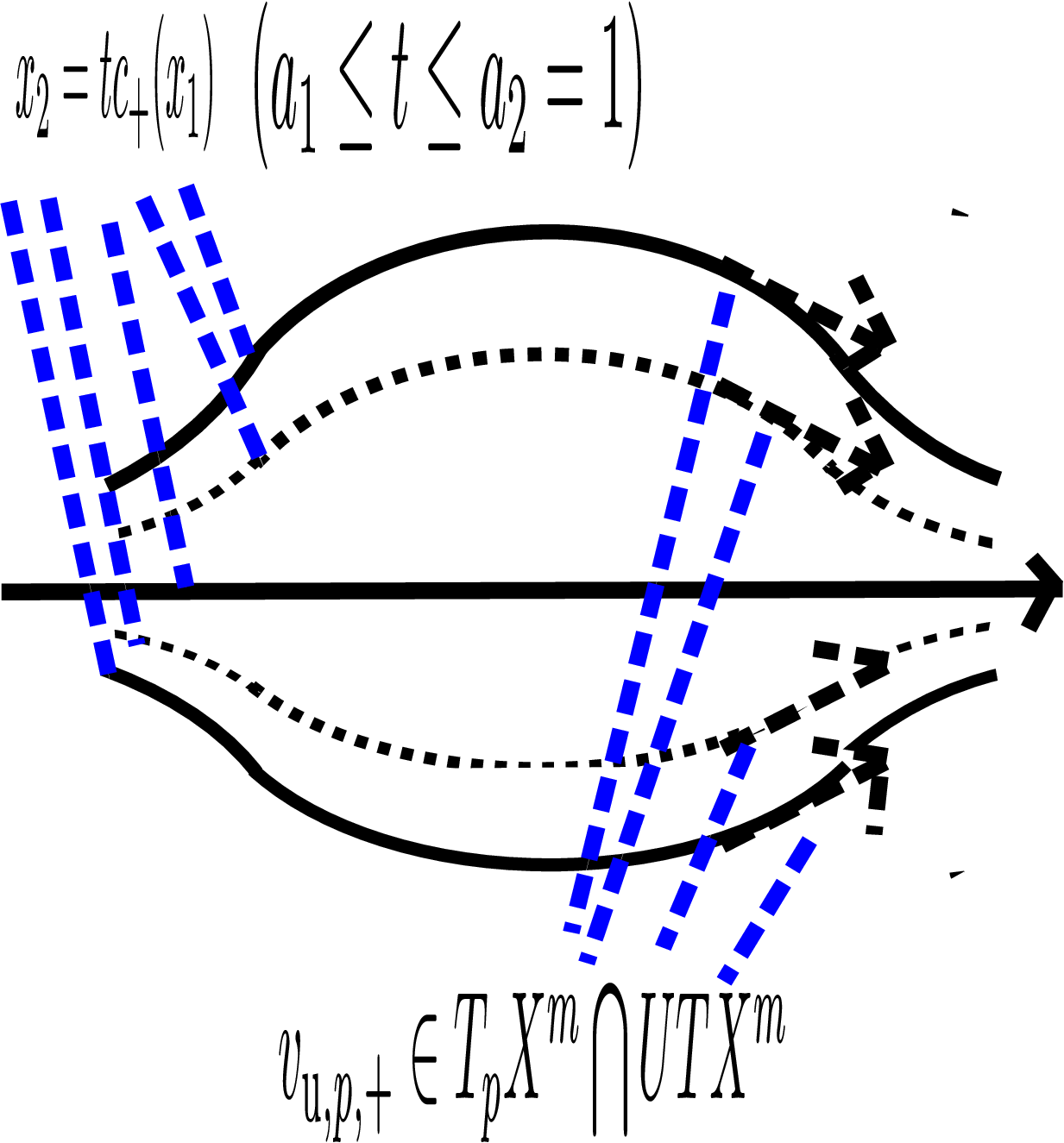}

\caption{The image of ${\pi}_{m+2,1} {\mid}_{X_{F_{D_{\{S_j\}_{j \in J}}},m}}$, the graphs  $\{(x_1,tc_{+}(x_1)) \mid x_1 \in \mathbb{R}\}$, and the tangent vectors $v_{{\rm u},p,+} \in T_p X^m \bigcap UTX^m$ in forms of the resulting vectors after the projection to ${\mathbb{R}}^2$. The tangent vectors in the forms of the vectors of ${\mathbb{R}}^2$ are all of length $1$ and tangent to graphs $\{(x_1,tc_{+}(x_1)) \mid x_1 \in \mathbb{R}\}$ at the points ${\pi}_{m+2,2}(p)$.}
\label{fig:1}
\end{figure}
We discuss an important example related to difference between Theorem \ref{thm:1} and Theorems \ref{thm:2} and \ref{thm:3}. 
\begin{Thm}
\label{thm:6}
In Theorems \ref{thm:4} and \ref{thm:5}, let $c_{+}:\mathbb{R} \rightarrow \mathbb{R}$ be an even smooth function whose value is always positive at any point $p \in \mathbb{R}$, which has the unique maximal value at $0 \in \mathbb{R}$, and which has no critical point except $0$.
\begin{enumerate}
\item We have a case of Theorem \ref{thm:1} by $n=2$, $l=2$ and $f_j$ and $S_j$ as follows{\rm :} $f_1(x_1,x_2):=c_{+}(x_1)-x_2$ with $S_1:=\{(x_1,x_2)\mid f_1(x_1,x_2)=0\}$ and $f_2(x_1,x_2):=x_2+c_{+}(x_1)$ with $S_2:=\{(x_1,x_2)\mid f_2(x_1,x_2)=0\}$.
\item The value of ${{\pi}_{m+1,1} {\mid}_{X_{D_{\{S_j,f_j\}_{j=1}^l},m}}}^{\prime}$ is $1$ at $p$ if and only if $p$ is of the form $(0,x_2,{(y_j)}_{j=1}^{m-1}) \in  X_{D_{\{S_j,f_j\}_{j=1}^l},m}$. The set of all such points is a smooth submanifold of  $X_{D_{\{S_j,f_j\}_{j=1}^l},m}$ and diffeomorphic to $S^{m-1}$.
\end{enumerate}
\end{Thm}
\begin{proof}
We use the symmetry.
Each vector $v_{{\rm u},p,+} \in T_p X^m \bigcap UTX^m$ with $X^m:=X_{F_{D_{\{S_j\}_{j \in J}}},m}$ in the definition above is represented as a vector of length $1$ mapped to $\mathbb{R}$, (rigorously, $T \mathbb{R}$,) by (the differential of) ${\pi}_{m+2,1} {\mid}_{X_{D_{\{S_j,f_j\}_{j=1}^l},m}}$. It is mapped to the unique vector of length $1$ if and only if $v_{{\rm u},p,+}$ is parallel to a vector whose 1st component is $1$ and whose remaining components are $0$ as a vector in ${\mathbb{R}}^{m+1}$. Note that rigorously, this is also seen as a vector of $T_p {\mathbb{R}}^{m+1}$. 
By our construction, the conditions on the global behavior of $c$ and $(c_{+}(x_1)-x_2)(x_2+c_{+}(x_1))$ (with the relation $(c_{+}(x_1)-x_2)(x_2+c_{+}(x_1))-{\Sigma}_{j=1}^{m-1} {y_j}^2=0$ for $(x_1,x_2,{(y_j)}_{j=1}^{m-1}) \in  X_{D_{\{S_j,f_j\}_{j=1}^l},m}$) and the derivation of $c_{+}$, and related symmetry, the vector $v_{{\rm u},p,+}$ is parallel to a vector whose 1st component is $1$ and whose remaining components are $0$ if and only if $p$ is of the form $(0,x_2,{(y_j)}_{j=1}^{m-1}) \in  X_{D_{\{S_j,f_j\}_{j=1}^l},m}$. These necessary and sufficient conditions complete the proof.
\end{proof}

Compare Theorem \ref{thm:6} to the case of Theorems \ref{thm:4} and \ref{thm:5} with $(a_1,a_2)=(-1,1)$.


\begin{Ex}
\label{ex:4}
For Theorem \ref{thm:6}, $c_{+}(x)=\frac{1}{x^2+1}$ is of simplest examples.
\end{Ex}

\section{Conflict of interest and Data availability.}
 The author is a researcher at Osaka Central Advanced Mathematical Institute (OCAMI researcher), funded by MEXT Promotion of Distinctive Joint Research Center Program JPMXP0723833165. He thanks members there for the hospitality, where he is not employed by the institute or the projects. 
 
No data other than the present file is generated, related to the present paper. We do not assume non-trivial arguments in preprints being still unpublished. To some extent, we may refer to these preprints.


\begin{thebibliography}{25}
\bibitem{buchstaberpanov} V. M. Buchstaber and T. E. Panov, \textsl{Toric topology}, Mathematical Surveys and Monographs, Vol. 204, American Mathematical Society, Providence, RI, 2015.

	%


\bibitem{burletderham} O. Burlet and G. de Rham, \textsl{Sur certaines applications g\'en\'eriques d'une vari\'et\'e close a $3$ dimensions dans le plan}, Enseign. Math. 20 (1974). 275--292.

\bibitem{eliashberg1} Y. Eliashberg, \textsl{On singularities of folding type}, Math. USSR Izv. 4 (1970). 1119--1134.
\bibitem{eliashberg2} Y. Eliashberg, \textsl{Surgery of singularities of smooth mappings}, Math. USSR Izv. 6 (1972). 1302--1326.
\bibitem{furuyaporto} Y. K. S. Furuya and P. Porto, \textsl{On special generic maps from a closed manifold into the plane}, Topology Appl. 35 (1990), 41--52.

\bibitem{gelbukh1} I. Gelbukh, \textsl{A finite graph is homeomorphic to the Reeb graph of a Morse-Bott function}, Mathematica Slovaca, 71 (3), 757--772, 2021; doi: 10.1515/ms-2021-0018. 
\bibitem{gelbukh2} I. Gelbukh, \textsl{Morse-Bott functions with two critical values on a surface}, Czechoslovak Mathematical Journal, 71 (3), 865--880, 2021; doi: 10.21136/CMJ.2021.0125-20. 
\bibitem{gelbukh3} I. Gelbukh, \textsl{On the topology of the Reeb graph}, Publicationes Mathematicae Debrecen 104(3--4) (2023), 343--365.


\bibitem{gelbukh4} I. Gelbukh, \textsl{Realization of a digraph as the Reeb graph of a Morse-Bott function on a given surface}, Topology and its Applications, 2024. 
\bibitem{gelbukh5} I. Gelbukh, \textsl{Reeb Graphs of Morse-Bott Functions on a Given Surface}, Bulletin of the Iranian Mathematical Society, Volume 50 Article number 84, 2024, 1--17. 


\bibitem{kitazawa1} N. Kitazawa, \textsl{On Reeb graphs induced from smooth functions on $3$-dimensional closed orientable manifolds with finitely many singular values}, Topol. Methods in Nonlinear Anal. Vol. 59 No. 2B, 897--912, arXiv:1902.08841.
\bibitem{kitazawa2} N. Kitazawa, \textsl{On Reeb graphs induced from smooth functions on closed or open surfaces}, Methods of Functional Analysis and Topology Vol. 28 No. 2 (2022), 127--143, arXiv:1908.04340.
\bibitem{kitazawa3} N. Kitazawa, \textsl{Real algebraic functions on closed manifolds whose Reeb graphs are given graphs}, Methods of Functional Analysis and Topology Vol. 28 No. 4 (2022), 302--308, arXiv:2302.02339, 2023. 

\bibitem{kitazawa4} N. Kitazawa, \textsl{Constructing Morse functions with given Reeb graphs and level sets}, accepted for publication in Topol. Methods in Nonlinear Anal., arXiv:2108.06913
(, where the title has been changed from the title there), 2025.
\bibitem{kitazawa5} N. Kitazawa, \textsl{Reconstructing real algebraic maps locally like moment maps with prescribed images and compositions with the canonical projections to the $1$-dimensional real affine space}, arXiv:2303.10723.
\bibitem{kitazawa6} N. Kitazawa, \textsl{Some remark on real algebraic maps which are topologically special generic maps and generalize the canonical projections of the unit spheres}, arXiv:2312.10646, 2024.
\bibitem{kitazawa7} N. Kitazawa, \textsl{A note on Reeb spaces of explicit real analytic functions}, a revised version is submitted to a refereed journal based on positive comments ("Major revision"), arXiv:2601.11648, 2026/1.

\bibitem{kitazawa8} N. Kitazawa, \textsl{Fundamental examples of Reeb spaces of smooth functions defined from two graphs of smooth functions with same asymptotic behaviors},  arXiv:2602.17014, 2026/2.

\bibitem{kitazawa9} N. Kitazawa, \textsl{Reeb spaces of functions being analytic on dense subsets and their graph structures}, arXiv:2602.23380.
\bibitem{kitazawa10} N. Kitazawa, \textsl{Reeb spaces of smooth functions associated to globally similar graphs of smooth functions}, arXiv:2603.02791.


\bibitem{masumotosaeki} Y. Masumoto and O. Saeki, \textsl{A smooth function on a manifold with given Reeb graph}, Kyushu J. Math. 65 (2011), 75--84.

\bibitem{michalak} L. P. Michalak, \textsl{Realization of a graph as the Reeb graph of a Morse function on a manifold}. Topol. Methods in Nonlinear Anal. 52 (2) (2018), 749--762, arXiv:1805.06727.
\bibitem{milnor1} J. Milnor, \textsl{Morse Theory}, Annals of Mathematic Studies AM-51, Princeton University Press; 1st Edition (1963.5.1).
\bibitem{milnor2} J. Milnor, \textsl{Lectures on the h-cobordism theorem}, Math. Notes, Princeton Univ. Press, Princeton, N.J. 1965.
\bibitem{reeb} G. Reeb, \textsl{Sur les points singuliers d\'{}une forme de Pfaff compl\'{e}tement int\`{e}grable ou d\'{}une fonction num\'{e}rique}, Comptes Rendus
 Hebdomadaires des S\'{e}ances de I\'{}Acad\'{e}mie des Sciences 222 (1946), 847--849.
\bibitem{saeki1} O. Saeki, \textsl{Notes on the topology of folds}, J. Math. Soc. Japan Volume 44, Number 3 (1992), 551--566.
\bibitem{saeki2} O. Saeki, \textsl{Topology of special generic maps of manifolds into Euclidean spaces}, Topology Appl. 49 (1993), 265--293, we can also find at "https://core.ac.uk/download/pdf/81973672.pdf" for example.
\bibitem{sharko} V. Sharko, \textsl{About Kronrod-Reeb graph of a function on a manifold}, Methods of Functional Analysis and
 Topology 12 (2006), 389--396.
\bibitem{thom} R. Thom, \textsl{Les singularites des applications differentiables}, Ann. Inst. Fourier (Grenoble) 6 (1955-56), 43--87.
\bibitem{whitney} H.  Whitney,  \textsl{On singularities of mappings of Euclidean spaces: I,  mappings of the plane into the plane},  Ann.  of Math.  62 (1955),  374--410. 

	
\end{thebibliography}
\end{document}